\documentclass[12pt]{amsart}
\usepackage{amssymb, amscd, amsmath, amsthm, latexsym, enumerate}

\renewcommand{\geq}{\geqslant}
\renewcommand{\leq}{\leqslant}

\newtheorem{theorem}{Theorem}
\newtheorem{lemma}[theorem]{Lemma}
\newtheorem{cor}[theorem]{Corollary}

\newtheorem*{cor*}{Corollary}

\begin{document}
\title{Non-solvable torsion free virtually solvable groups}

\author{Jonathan A. Hillman }
\address{School of Mathematics and Statistics\\
     University of Sydney, NSW 2006\\
      Australia }

\email{jonathanhillman47@gmail.com}

\begin{abstract}
There are perfect Bieberbach groups of Hirsch length 15,
but none in lower dimensions.
We shall show that a non-solvable, torsion free, 
virtually solvable group $S$ must have Hirsch length $h(S)\geq10$.
If $h(S)\leq13$ then we may assume that $A_5$ is the only simple factor,
but $PSL(2,7)$ and $SL(2,8)$ may occur when $h(S)\geq14$.
There are no known examples with $h(S)<15$.
\end{abstract}



\maketitle

We shall consider here the question:
{\sl What is the smallest torsion free virtually solvable group 
which is not solvable?}
Here ``smallest" should be interpreted as having minimal Hirsch length.
Plesken showed that 15 is the smallest dimension in which there are 
non-solvable Bieberbach groups \cite{Pl89},
and Lutowski and Szczepa\'nski have given two explicit examples 
of such groups,
one with holonomy $A_5=PSL(2,5)$ and another with holonomy $PSL(2,7)$ \cite{LuS}.
The question remains open for groups which are not virtually abelian. 
We may assume that $S$ is finitely generated and perfect, 
and that the quotient of $S$ by its maximal solvable normal subgroup $\widetilde{S}$ 
is a minimal simple group.
We shall show that if there is such a group $S$ with Hirsch length $h(S)<15$
then $h(S)\geq10$ and $S/\widetilde{S}\cong{A_5}$,  $PSL(2,7)$ or $SL(2,8)$.
In the latter two cases $h(S)\geq14$. 
If $S$ is virtually nilpotent then its Fitting subgroup has nilpotency class $\leq3$.
Our arguments rest upon the groups in question having finite perfect quotients 
which act effectively on a free abelian group of small rank. 
We shall not consider whether there are examples of Hirsch length 15 
other than Bieberbach groups.

The first section is on notation and terminology,
and the next section contains some basic results on crystallographic quotients
of finitely generated virtually solvable groups.
In \S3 we introduce the notion of (minimal) TFNS group,
and in \S4 we use knowledge of the finite subgroups of $GL(k,\mathbb{Z})$ 
and $Sp(2\ell,\mathbb{Z})$ for $k$ and $\ell$ small \cite{Ki, PP}
to find the relevant minimal perfect groups and their representations.
In the next section we show that if $G$ is a finitely generated perfect subgroup of 
$GL(6,\mathbb{Q})$ which is an extension of a simple group $H$ 
by a solvable normal subgroup then $H\cong{A_5}$ or $PSL(2,7)$.
In \S6 we use some commutative algebra to show that if a torsion free,
virtually solvable group $S$ is neither solvable nor virtually nilpotent then $h(S)\geq9$. 
In \S7 we consider the commutator pairing for nilpotent groups,
and we apply this work in \S8 to torsion free virtually nilpotent groups 
which are not solvable.
The condition that $S$ be torsion free is used in some of the early lemmas and in \S6, 
but is most prominent in \S9,
where we show that $h(S)\geq10$, 
and limit the structure of the Fitting subgroup when $S$ is virtually nilpotent.
The final brief section contains a few questions.

I would like to thank R. Lutowski for his advice on representations of
finite perfect groups and on the associated crystallographic groups,
and D. Taylor for pointing out gaps in my first attempt at proving Theorem \ref{no sl(2,8)},
which prompted rewriting \S5.

\section{generalities}

If $G$ is a group then $|G|$, $G'$, $X^2(G)$ and $ \zeta{G}$ shall denote the order,
commutator subgroup,  subgroup generated by squares 
and centre of $G$, respectively,
and $C_n$ shall denote the cyclic group of order $n$.
(Thus $G'\leq{X^2(G)}$, since $G/X^2(G)$ is an elementary abelian 2-group.)
Let $I(G)$ be the {\it isolator\/} of $G'$ in $G$,
so that $G/I(G)$ is the maximal torsion free abelian quotient of $G$.
If $J$ is a subgroup of $G$ then $C_G(J)$ is the centralizer of $J$ in $G$.

Let $G^{(0)}=G$ and $G^{(n+1)}={G^{(n)}}'$ be the terms of the derived series for $G$, 
and let $\gamma_1G=G$ and $\gamma_{n+1}G=[G,\gamma_nG]$ be the terms
of the lower central series for $G$.
(This notation is from \cite{Ro}.)
Let $\sqrt{G}$ denote the Hirsch-Plotkin radical of $G$.
If $G$ is virtually polycyclic then $\sqrt{G}$ is the unique maximal nilpotent
normal subgroup of $G$, and is also known as the Fitting subgroup of $G$.

A finite perfect group is {\it minimal\/} if it is nontrivial and
all of its proper subgroups are solvable.
The minimal simple groups are $PSL(2,q)$ for $q=2^p$ with $p$ a prime,
$q=3^\ell$ with $\ell$ odd, $q=p$ a prime such that $p^2\equiv-1$ mod (5),
a Suzuki group $Sz(2^p)$ with $p$ an odd prime or $PSL(3,3)$ \cite{Th}.
(We shall use ``simple"  to mean ``non-abelian and simple" throughout.)
Every finite simple group contains a minimal simple group \cite{BW}.
(This is not entirely obvious from the definitions!)
There appears to be no corresponding determination of minimal perfect groups.

A virtually solvable group $G$ has a maximal solvable normal subgroup 
$\widetilde{G}$ of finite index.
If $G$ is not solvable then the lowest term of a composition series 
for $G/\widetilde{G}$ is a finite simple group, 
and so contains a minimal simple group \cite{BW}.

An {\it outer action\/} of a group $H$ on a group $G$ 
is a homomorphism $\theta:H\to{Out(G)}$.
Such a homomorphism induces homomorphisms from $H$ to $Aut(G/G')$ 
and $Aut(\zeta{G})$.
We shall say that the action $\theta$ is {\it effective\/} if it is a monomorphism.

A {\it crystallographic\/} group $G$ is a group which is an extension $\xi$
of a finite group $H$ by a finitely generated free abelian group $A$, 
such that the action $\theta:H\to{Aut(A)}$ induced by conjugation 
in $G$ is effective.
The {\it holonomy group\/} of $G$ is the quotient $G/A$.
Such extensions are classified by $\theta$ and a cohomology class
$c(\xi)$ in $H^2(H;A^\theta)$, where $H$ acts on $A$ via $\theta$.
(We shall write just $H^2(H;A)$, when the action is clear.)
The semidirect product $A\rtimes_\theta{H}$ corresponds to $c(\xi)=0$.

A {\it Bieberbach\/} group is a torsion free crystallographic group.
A crystallographic group corresponding to an extension $\xi$
is a Bieberbach group if and only if  $c(\xi)$ has non-zero restriction
in $H^2(C;A)$ for every cyclic subgroup $C<H$ of prime order $>1$
\cite[Theorem 3.1]{Sz}.

While the determination of all the finite subgroups of $GL(n,\mathbb{Z})$
requires considerable computation once $n>3$ \cite{CS,PP},
it is quite easy to estimate the maximal possible orders of such subgroups.
We shall summarize the account in \cite{KP}.
If $p$ is a prime then the maximal power of $p$ dividing the order 
of a finite subgroup $G<GL(n,\mathbb{Q})$ is  
$e_n(p)=\Sigma_{j\geq0}\lfloor{\frac{n}{p^j(p-1)}}\rfloor$.
In particular,  $e_n(p)=0$ if $p>n+1$,
while if $n<12$ then $e_n(7)\leq1$ and $e_n(5)\leq2$,
and if $12\leq{n}\leq14$ then $e_n(11)=e_n(13)=1$, 
$e_n(7)=2$ and $e_n(5)=3$.
Moreover, if $G$ is cyclic and of prime power order $p^k\geq3$ then 
$p^{k-1}(p-1)\leq{n}$.
If $G$ is cyclic of composite order and $|G|\not\equiv2~mod~(4)$ then the sum 
of such terms corresponding to the prime powers dividing the order is at most $n$, 
while if $|G|\equiv2~mod~(4)$ this sum is at most $n+1$ \cite{KP}.

\section{crystallographic quotients}

We have adapted the next two lemmas for our needs.

\begin{lemma}
\label{crys}
\cite[Theorem 2.2]{Sz}
Let $G$ be a finitely generated, virtually abelian group.
Then $G$ is a crystallographic group if and only if it has no non-trivial finite normal subgroup.
In that case $\sqrt{G}$ is  free abelian and is
the maximal normal abelian subgroup of $G$,
and $G$ has holonomy $H=G/\sqrt{G}$.
If $H$ is not solvable then $h(G)\geq4$.
\end{lemma}

\begin{proof}
We may assume that $G$ is an extension of a finite group $H$ 
by a finitely generated free abelian normal subgroup $A$.
Then $A$ has finite index in $\sqrt{G}$, 
and so $\sqrt{G}$ is nilpotent and has finite torsion subgroup.

Let $C=C_G(A)$ be the centralizer of $A$ in $G$.
Then $G/C$ acts effectively on $A$.
Since $[C:A]$ is finite, $C'$ is finite,
by a theorem of Schur \cite[10.1.4]{Ro}.
Hence if $G$ has no non-trivial finite normal subgroup then
$C'=1$ and $C$ is free abelian.
By the same reasoning, $\sqrt{G}$ is abelian and so $C=\sqrt{G}$
is the maximal abelian normal subgroup of $G$.
Since $G/\sqrt{G}=G/C$ acts effectively on $A$ it acts effectively on $\sqrt{G}$.
Thus $G$ is a crystallographic group with holonomy $H_1=G/\sqrt{G}$.

Conversely,  suppose that $G$ is a crystallographic group with holonomy $H$.
If $F$ is a finite normal subgroup of $G$ then $A\cap{F}=1$,
since $A$ is torsion free, and so $F$ projects injectively to $H$.
Moreover, since $A$ and $F$ are each normal,
$[A,F]\leq{A}\cap{F}=1$ and so $AF\cong{A}\times{F}$.
Since the action of $H$ on $A$ is effective 
we have $F=1$.

The final observation holds since finite subgroups of 
$GL(3,\mathbb{Q})$ have order dividing $48$, 
and so  are solvable.
\end{proof}

In this context the Fitting subgroup $\sqrt{G}$ is also called the 
{\it translation\/} subgroup of $G$.

\begin{lemma}
\label{HS?}
\cite[Prop.  4.1]{Sz}
Let $G$ be a crystallographic group with holonomy $H$.
If $H^1(G;\mathbb{Q})=0$ then ${Hom(\sqrt{G},\mathbb{Q})}^H=1$.
Hence the $\mathbb{Q}[G]$-module $\mathbb{Q}\otimes\sqrt{G}$
has no summand $\mathbb{Q}$ with trivial $H$-action.
\end{lemma}

\begin{proof}
This follows from the exact sequence of low degree in the 
LHS spectral sequence for $G$ as an extension of $H$ by $\sqrt{G}$,
since $H^1(G;\mathbb{Q})=H^2(H;\mathbb{Q})=0$.
The second assertion is then clear.
\end{proof}

Let $N$ be a nilpotent group,
and let $tN$ be its torsion subgroup.
Then $N/tN$ is torsion free and has a  central series 
with torsion free abelian subquotients \cite[5.2.7 and 5.2.20]{Ro}.
Let 
\[
\tilde\gamma_1N=N>\dots>\tilde\gamma_cN=tN
\]
be the preimage in $N$ of the most rapidly descending such central series. 
Then $\tilde\gamma_2N=I(N)$.
Let $\mathbb{Q}N^{ab}=\mathbb{Q}\otimes{N/N'}=
\mathbb{Q}\otimes{N/\tilde\gamma_2N}$.

\begin{lemma}
\label{vnil}
Let $G$ be a torsion free virtually solvable group such that $h(G)$ is finite.
If $G/\sqrt{G}$ is a torsion group then the homomorphism from $G/\sqrt{G}$ 
to $Aut(\mathbb{Q}\sqrt{G}^{ab})$ induced by conjugation in $G$ 
is a monomorphism and  $G/\sqrt{G}$ is finite.
If $G$ is finitely generated then $G/I(\sqrt{G})$ is crystallographic.
\end{lemma}

\begin{proof}
Since $r=h(\sqrt{G}^{ab})\leq{h(\sqrt{G})}=h(G)$,  it is finite.
Let $C$ be the kernel of the homomorphism from $G$ to
$Aut(\mathbb{Q}\sqrt{G}^{ab})\cong{GL(r,\mathbb{Q})}$ induced by conjugation in $G$.
Then $\sqrt{G}\leq{C}$ and $C/I(\sqrt{G})$ contains 
$\sqrt{G}/I(\sqrt{G})$ as a central subgroup of finite index.
Conjugation by elements of $C$ also induces the identity on 
the subquotients $\tilde\gamma_i\sqrt{G}/\tilde\gamma_{i+1}\sqrt{G}$, for all $i>1$.
Hence $C$ is nilpotent, 
by Baer's extension of a theorem of Schur \cite[14.5.1]{Ro}
and the fact that $C$ is torsion free.
Hence $C=\sqrt{G}$, by the maximality of $\sqrt{G}$,
and so the homomorphism induced by conjugation is a monomorphism.
Since the image of $G/\sqrt{G}$ is a torsion subgroup of $GL(r,\mathbb{Q})$
 it is finite \cite[8.1.11]{Ro}.
This monomorphism factors through $Aut(\sqrt{G}/I(\sqrt{G}))$,
and so $G/\sqrt{G}$ acts effectively on $\sqrt{G}/I(\sqrt{G})$.
If $G$ is finitely generated then $\sqrt{G}/I(\sqrt{G})$ 
$\cong\mathbb{Z}^r$, 
and so $G/I(\sqrt{G})$ is a crystallographic group.
\end{proof}

\begin{lemma}
\label{HPnilp}
Let $G$ be a torsion free solvable group such that $h(G)$ is finite.
Then $\sqrt{G}$ is nilpotent, and $C_G(\sqrt{G})=\zeta\sqrt{G}$.
\end{lemma}

\begin{proof}
If $N$ is a  finitely generated subgroup of $\sqrt{G}$ then $h(N)\leq{h}=h(G)$,
and so $\gamma_{h+1}N=1$, since $G$ is torsion free.
It follows immediately that $\gamma_{h+1}\sqrt{G}=1$, 
and so $\sqrt{G}$ is nilpotent.

Let $C=C_G(\sqrt{G})$, 
and suppose that $\sqrt{G}$ is a proper subgroup of $C.\sqrt{G}$.
Since $G$ is solvable it has a normal subgroup $D$
such that $\sqrt{G}<D\leq{C.\sqrt{G}}$ and $D/\sqrt{G}$ is abelian.
But then $D$ is nilpotent, 
contradicting the maximality of the Hirsch-Plotkin radical.
Hence $C=C_G(\sqrt{G})\leq\sqrt{G}$ and so $C=\zeta\sqrt{G}$.
\end{proof}

In the virtually polycyclic case $G/\sqrt{G}$ is virtually abelian \cite[15.1.6]{Ro}.
However, this not so for all finitely presentable solvable groups \cite{RS}.

If $G$ is virtually polycyclic and all of its abelian subnormal subgroups  
have rank $\leq{n}$ then $G/\sqrt{G}$ is virtually abelian of rank $<n$  
\cite[Theorem 2]{Wi}.
Hence $h(G/\sqrt{G})<h(\sqrt{G})$,
since $\sqrt{G}$ contains all abelian subnormal subgroups.
We may push this inequality a little further.

\begin{lemma}
\label{Wilson}
Let $G$ be a virtually polycyclic group, and let $V$ be a 
normal subgroup of finite index which contains $\sqrt{G}$
and such that $V/\sqrt{G}$ is abelian.
Let $C$ be the preimage in $G$ of the centre of $V/\sqrt{G}'$.
Then $h(G/\sqrt{G})<h(\sqrt{G}/\sqrt{G}'.V'\cap{C})$.
\end{lemma}

\begin{proof}
We note first that $\sqrt{G}=\sqrt{V}$,
since $V$ is a normal subgroup and $\sqrt{G}\leq{V}$.
Hence we may assume that $G=V$, since $h(G)=h(V)$,
and so $G'\leq\sqrt{G}$.
The preimage of $\sqrt{G/\sqrt{G}'}$ in $G$ is $\sqrt{G}$
\cite[5.2.10]{Ro},
and so we may pass to the quotient $G/\sqrt{G}'$.
Hence we may also assume that $\sqrt{G}$ is abelian.
We then have $C=\zeta{G}\leq\sqrt{G}$ and
$\sqrt{G/(G'\cap{C})}=\sqrt{G}/(G'\cap{C})$,
since $G'\cap{C}$ is central.
Applying \cite[Theorem 2]{Wi} to $G/G'\cap{C}$ gives the result.
\end{proof}

This bound is sharp.
For example, if $K$ is a totally real number field of degree 5, 
then the group of integral units $\mathcal{O}_K^\times$ has rank 4,
and acts effectively on $\mathcal{O}_K$.
Hence $G=\mathcal{O}_K\rtimes\mathcal{O}_K^\times$ is 
virtually torsion free poly-$\mathbb{Z}$.
The abelian normal subgroup $\mathcal{O}_K$ is its own centralizer in $G$,
and so $\sqrt{G}=\mathcal{O}_K$ and $\zeta{G}=1$.
Hence  $h(G/\sqrt{G})=4<h(\sqrt{G})=5$.

The hypothesis that $G$ be virtually polycyclic is necessary.
If $m>1$ the Baumslag-Solitar group $G$ with presentation
$\langle{a,t}\mid{tat^{-1}=a^m}\rangle$ is solvable,
but is not polycyclic since
$\sqrt{G}=I(G)\cong\mathbb{Z}[\frac1m]$ is not finitely generated.
In this case $h(G/\sqrt{G})=h(\sqrt{G})=1$.

\section{minimal TFNS groups}

A group $G$ is {\it TFNS\/} if it is torsion free and virtually solvable but not solvable.
Since $G/\widetilde{G}$ is finite, 
$G$ has a finitely generated subgroup with the same
nonsolvable finite quotient.
A TFNS group t is {\it minimal\/} if it is finitely generated,  
$G/\widetilde{G}$ is a minimal simple group and $G$
has minimal Hirsch length for such groups.
(Cf. \cite[Definition 2.1]{LuS}.)
If $G$ is minimal TFNS then $h(G)\leq15$,
by the observations in the first paragraph of the introduction.
We shall assume henceforth that $S$ is a torsion free virtually solvable group
which is not virtually abelian.

If  $S$ is minimal TFNS then $S^{ab}$ is finite.
For otherwise there would be an epimorphism $\phi:S\to\mathbb{Z}$, 
and $\mathrm{Ker}(\phi)$ would not be solvable.
Since $\mathrm{Ker}(\phi)$ is virtually solvable, 
it has finitely generated subgroups which are non-solvable,  
but have Hirsch length $<h(S)$,
contradicting the minimality of $S$.
We can improve on this.

\begin{lemma}
\label{perf}
Let $G$ be a virtually solvable group such that $H=G/\widetilde{G}$ is perfect.
If $\widetilde{G}^{(n)}=1$ then $G^{(n)}$ is perfect,
and if $H$ is simple then $G^{(n)}/\widetilde{G^{(n)}}\cong{H}$.
\end{lemma}

\begin{proof}
Let $f:G\to{T}=G/G^{(n+1)}$ be the natural epimorphism.
Then $f(\widetilde{G})$ is normal in $T$,
and $f$ induces an epimorphism from $H$ onto $T/f(\widetilde{G})$.
Since $H$ is perfect and $T$ is solvable it follows that $f(\widetilde{G})=T$,
and so $f$ induces epimorphisms $f^{(k)}:\widetilde{G}^{(k)}\to{T^{(k)}}$, 
for all $k\geq1$.
Hence $T^{(n)}=1$ and so $G^{(n)}=G^{(n+1)}$.
Thus $G^{(n)}$ is perfect.
The final assertion is clear, 
since $H$ is the only non-solvable quotient of $G$.
\end{proof}

Specializing further, we see that if $G$ is a crystallographic group 
with perfect holonomy $H$ then $G'$ is a perfect crystallographic group,
and if $H$ is simple then $G'$ also has holonomy $H$.

\begin{lemma}
\label{both}
Let $S$ be a minimal TFNS group and let $H=S/\widetilde{S}$.
Then $S$ has a perfect subgroup of finite index which maps onto $H$.
\end{lemma}

\begin{proof}
By Lemma \ref{perf},
there is an $n\geq0$ such that $S^{(n)}$ is perfect,
and  $S^{(n)}/\widetilde{S^{(n)}}\cong{H}$, 
Hence $h(S^{(n)})=h(S)$, by minimality of $S$,
and so $h(S/S^{(n)}=0$.
Since $S$ is finitely generated,
$S/S^{(n)}$ is finite,
and hence also finitely generated.
\end{proof}

Thus we may assume that a minimal TFNS group is also perfect.

\begin{lemma}
\label{overline}
Let $S$ be a minimal TFNS group which is not virtually abelian.
Then $S$ has normal subgroups $V\leq{U}\leq\widetilde{S}$ such that
$S/U$ is crystallographic,  $U/V$ is finite, $V/I(V)$ has positive rank
and $U/V$ acts effectively on $V/I(V)$.
\end{lemma}

\begin{proof}
Since $S$ is finitely generated, virtually solvable and infinite 
it has a normal subgroup $T$ of finite index which is solvable 
and has infinite abelianization.
Let $G=S/I(T)$, $A=T/I(T)$ and and $B=C_G(A))$.
Then $A\leq\zeta{B}$ and $[B:A]$ is finite, 
and so $B'$ is finite \cite[10.1.4]{Ro}.
Hence also $I(B)$ is finite, 
and $A$ embeds in $B/I(B)$.
Since $G/B$ acts effectively on $A$, it also acts effectively on $B/I(B)$.

Let $U$ and $W$ be the preimages of $I(B)$ and $B$ in $S$.
Then $U$ and $W$ are normal subgroups of $S$ and $h(U)<h(S)$.
Hence $U$ is solvable,  by minimality of $S$, and so $W$ is also solvable.
Moreover $S/W\cong{G/B}$ acts effectively on $W/U\cong{B/I(B)}$,
so $S/U$ is crystallographic.

Since $U\leq\widetilde{S}$ it is solvable, 
and since $S$ is not virtually abelian $U\not=1$.
Let $U^{(i)}$ be the largest term in the derived series for $U$ such that
$U^{(i)}/I(U^{(i)})$ has positive rank,
and let $V$ be the preimage in $U$ of the centralizer of $U^{(i)}/I(U^{(i)})$ 
in $U/I(U^{(i)})$. 
Then $U^{(i)}\leq{V}\leq{U}$, and so $U/V$ and $V/U^{(i)}$ are torsion groups. 
The group $U^{(i)}/I(U^{(i)})$ is a torsion free abelian group of finite rank,
since $h(S)$ is finite, and $U/V$ acts effectively on $U^{(i)}/I(U^{(i)})$.
Hence $U/V$ is finite \cite[8.1.11]{Ro}.
Let $G=V/I(U^{(i)})$.
Since $U^{(i)}/I(U^{(i)})$ is central in $G$ and $V/U^{(i)}$ is a torsion group,
the commutator subgroup $G'\cong{V'I(U^{(i)})/I(U^{(i)})}$ is also a torsion group.
Hence $U^{(i)}/I(U^{(i)}$ embeds in $V/I(V)$,
and $h(V/I(V))=h(U^{(i)}/I(U^{(i)})>0$.
It also follows easily that $U/V$ acts effectively on $V/I(V)$.
\end{proof}

\begin{cor}
If $S$ is perfect then either $U=V$ or $Aut(V/I(V))$ has a nontrivial finitely 
generated perfect subgroup.
\end{cor}

\begin{proof}
If $S$ is perfect then so is $S/I(V)$, 
and the image of $S/V$ in $Aut(V/(IV))$ contains the image of $U/V$.
\end{proof}

If $S$ is virtually nilpotent we may take $V=U=I(\sqrt{S})$,
by Lemma \ref{vnil},
and if $S$ is virtually polycyclic but not virtually nilpotent
then we may take $V=\sqrt{S}$ (cf. \cite[15.1.6]{Ro}).

\begin{cor}
If $S$ is minimal TFNS then $h(S/U)\geq4$. 
\end{cor}

\begin{proof}
If $S$ is minimal TFNS then $S/U$ is not solvable,
since $U$ is solvable.
Hence $h(S/U)\geq4$, by Lemma \ref{crys}.
\end{proof}

We shall invoke the hypotheses and notation of Lemma \ref{overline} frequently
in \S5 and \S8 below.

\section{the relevant minimal perfect groups}

We are interested in minimal perfect groups $H$ which have non-trivial 
homomorphisms to $GL(n,\mathbb{Z})$, for some $n\leq14$. 
We may in fact work with coefficients $\mathbb{Q}$,
as every finite subgroup of $GL(n,\mathbb{Q})$ is conjugate into 
$GL(n,\mathbb{Z})$. 
Moreover,  the rational group ring $\mathbb{Q}[H]$ is semisimple,
which simplifies our analysis.
The image of $H$ in $GL(n,\mathbb{Q})$ is again a minimal perfect group, 
but may be a proper quotient of $H$.
(When $S$ is virtually nilpotent and $N=\sqrt{S}$,
we may assume that $H=S/N$ embeds in $Aut(\mathbb{Q}N^{ab})$,
by Lemma \ref{vnil}.)

There is a further simplification.
If $h(S)\leq14$ and $h(S/U)>10$ then $h(V)\leq3$,
and we may use Lemma \ref{sl2} below to show that $S$ acts trivially on 
the abelian sections of $V$.
Hence $S$ is virtually nilpotent, and $h(I(\sqrt{S})\leq3$.
We may assume that $S$ is not virtually abelian.
It then follows from Lemma \ref{Sp12} below that $S/\sqrt{S}$ 
preserves an antisymmetric pairing on $\mathbb{Q}\sqrt{S}^{ab}$
which has rank $\leq13$.
Thus the groups of interest to us are either subgroups of $GL(10,\mathbb{Q})$
or subgroups of the symplectic group $Sp(12,\mathbb{Q})$.

It follows easily from the material from \cite{KP} at the end of \S1 
that the only prime powers dividing the orders of finite subgroups 
of $GL(10,\mathbb{Q})$ are 2,3,4,5,7,8,9,11, 13 and 16.
The projective linear groups $PSL(2,q)$ contain cyclic subgroups 
of order $q-1$, $q$ and $q+1$, 
and so we may eliminate such groups for which $q=2^p$ with $p>3$
or $q=3^\ell$ with $\ell>3$.
The Suzuki group $Sz(2^p)$ has a cyclic subgroup of order $2^p+2^s+1$,
where $2s=p+1$, and so we may eliminate such groups with $p>3$.
This leaves only $A_5=PSL(2,5)=PSL(2,4)$, $PSL(2,7)$,  $SL(2,8)=PSL(2,8)$,
$PSL(2,13)$,  $PSL(2, 27)$,  $Sz(8)$ and $PSL(3,3)$.
The last four groups have elements of order 13.

However we still need to consider extensions of such groups by 
solvable normal subgroups.
We shall simplify our task by using the work of Lutowski and Szczep\'anski,
who show that the minimal perfect groups with irreducible embeddings in
$GL(n,\mathbb{Q})$ (for some $n\leq10$) are:
$A_5$, $PSL(2,7)$,  $SL(2,8)$,
the central extensions $SL(2,5)$ and $SL(2,7)$,
and $L_3(2)N2^3$, the non-split extension of $PSL(2,7)$ by $C_2^3$ \cite{LuS}.

If $H$ is simple then all non-trivial representations are faithful.
We shall label the non-trivial $\mathbb{Q}$-irreducible $\mathbb{Q}$-rational 
characters of the groups of most interest to us by their degree. 
We shall also use the same symbols to denote the associated $\mathbb{Q}[H]$-modules.
(See \cite{HP}.)

$A_5$ has three: $\rho_4,\rho_5,\rho_6$.

$PSL(2,7)$ has four: $\tau_{6a},\tau_{6b},\tau_7,\tau_8$.

$SL(2,8)$ also has four: $\psi_7,\psi_8,\psi_{21},\psi_{27}$.\\
The other three groups are not simple, 
but have faithful representations in dimensions 7 or 8.

$SL(2,5)$ has two: $\pi_{8a}$ and $\pi_{8b}$.

$SL(2,7)$ has one: $\xi_8$.
 
$L_3(2)N2^3$ has two: $\lambda_{7a}$ and $\lambda_{7b}$.

We shall also consider faithful representations which are reducible, 
but have no trivial summands.
In particular, $SL(2,5)$ has two such representations in each of dimensions 12 and 13,
given by $\pi_{8i}\oplus\widehat\rho_j$, where $i=a$ or $b$ and $j=4$ or 5,
and $\widehat\rho_j$ is the representation which factors through $\rho_j$.
Similarly, $L_3(2)N2^3$ has four such representations in dimension $13$,
but $SL(2,7)$ has none in dimensions $<14$.

In the sections below on nilpotent groups we shall need to consider also
{\it symplectic\/} representations.
A representation $\rho$ into  $GL(2k,\mathbb{Q})$ is symplectic if it is conjugate
into the subgroup $Sp(2k,\mathbb{Q})$.

\begin{lemma}
\label{symplectic}
Let $\rho:F\to{GL(2k,\mathbb{Q})}$ be an irreducible representation of 
a finite group $F$. 
Then $\rho$ is symplectic if and only if the trivial representation $\mathbf{1}$ is a summand of the exterior square $\wedge_2\rho$.
\end{lemma}

\begin{proof}
If $\rho$ is symplectic then the associated skew-symmetric pairing defines 
a non-zero $F$-linear homomorphism from $\wedge_2\mathbb{Q}^n$ to $\mathbb{Q}$.
Since $\mathbb{Q}[F]$ is semisimple, 
$\mathbf{1}$ is a summand of $\wedge_2\rho$.

Conversely,  a projection from $\wedge_2\rho$ onto $\mathbf{1}$ gives a skew-symmetric pairing on $\wedge_2\mathbb{Q}^n$. 
Since $\rho$ is irreducible the radical of this pairing is 0, 
and so the pairing is non-singular.
Hence $\rho$ is symplectic.
\end{proof}

The finite subgroups of $Sp(2k,\mathbb{Q})$ are determined in \cite{Ki},
for $2k\leq12$. 
In particular, all such groups with order divisible by 13 are solvable,
and the representation of $L_3(2)N2^3$ in $GL(8,\mathbb{Q})$ is not symplectic.

The following information on torsion was provided by R. Lutowski.
Let $G$ be a crystallographic group with holonomy $H$.
If $H\cong{A_5}$ and $h(G)=4$ or 5 then $G$ has 5-torsion.
If $H\cong{PSL(2,7)}$ and $h(G)=6$ or 7,
or if $H\cong{SL(2,8)}$ and $h(G)=7$,
or if  $H\cong{SL(2,7)}$ and $h(G)=8$,
or if $H\cong{L_3(2)N2^3}$ and $h(G)=7$ or 8 then $G$ has 7-torsion.
There are examples with $H\cong{PSL(2,7)}$ or $SL(2,8)$ and $h(G)=8$
with no 7-torsion.  (See also \cite[Chapter 6]{HP}.)

\newpage
\section{perfect virtually solvable subgroups 
of $GL(6,\mathbb{Q})$}

The main result of this section is that if $G$ is a finitely generated, 
perfect, virtually solvable subgroup of $GL(6,\mathbb{Q})$ such that 
$G/\widetilde{G}$ is simple then $G/\widetilde{G}\cong{A_5}$ or $PSL(2,7)$

\begin{lemma}
\label{sl2}
Let $G$ be a finitely generated, virtually solvable,  
perfect subgroup of $GL(d,\mathbb{Q})$. 
If $d\leq3$ then $G=1$,  if $d\leq5$ then $|G/\widetilde{G}|$ divides
$2^{10}3^25$, and if $d=6$ then $|G/\widetilde{G}|$ divides $2^{12}3^45.7$.
\end{lemma}

\begin{proof}
Since $G$ is perfect, it is a subgroup of $SL(d,\mathbb{Q})$.
Let $m$ be the lowest common denominator for 
the entries of a generating set for $G$,
and let $R=\mathbb{Z}[\frac1m]$.
Then $G\leq{SL(d,R)}$.
Let $p>m$ be an odd prime,
and let $G_k$ be the kernel of the projection of $G$ into $SL(d,R/p^kR)$.
Then $G_1$ is torsion free,
$G_k/G_{k+1}$ is an elementary $p$-group for all $k\geq1$,
and $\cap_{k\geq1}G_k=1$.
If $G\not=1$ then $G_1$ is solvable and $G/G_1$ is a perfect subgroup of $SL(d,p)$.
Hence $G/\widetilde{G}$ has order dividing $|SL(d,p)|$.

It is easy to see that $|SL(d,p)|=p^{\binom{d}2}\Pi_{i=2}^d(p^i-1)$
 \cite[3.2.7]{Ro}, 
and so $|SL(d,p)|$ is divisible by all primes $\leq{d+1}$.

If $d=3$ then $|SL(3,p)|=p^3(p-1)^2(p^2+p+1)(p+1)$.
If $p\equiv3~mod~(8)$ then this is divisible by $2^4$ but not by $2^5$,
while if $p\equiv2~mod(9)$ then it is divisible by 3 but not by $3^2$.
If $q$ is a prime $>3$ and $p\equiv2~mod~(q)$ then the factors are congruent 
to $2^6,1,7$ and 3, respectively,  while if $p\equiv3~mod~(q)$ 
the factors are congruent to $3^6,4,13$ and 4, respectively.
Using the infinitude of primes in arithmetic progressions,
we find that the highest common factor of such orders (taken over all $p>m$) is 48.
Hence $G$ has no simple quotient.
In this case $G$ is a solvable perfect group and so $G=1$.

If $d=5$ or 6 similar arguments show that the highest common factors 
of such orders are $2^93^25$ and $2^{12}3^45.7$,
respectively.
\end{proof}

The first assertion is sharp.
The action of $A_5$ on $\mathbb{Z}^5$ by permuting the coordinates fixes
the hyperplane $\sum{x_i}=0$.
Hence $A_5$ acts effectively on $\mathbb{Z}^4$.
The corresponding semidirect product $\mathbb{Z}^4\rtimes{A_5}$ 
is crystallographic of Hirsch length 4,  but is not solvable.

Minimal perfect groups of order dividing $2^93^25$ are extensions of $A_5$ 
by 2-groups \cite{HP}. 
There are many more perfect groups of order dividing $2^{12}3^45.7$,
but $A_5$, $PSL(2,7)$ and $SL(2,8)$ are the only such minimal simple groups.

The next two lemmas shall be used to exclude the possibility that 
$GL(6,\mathbb{Q})$ has such a subgroup $G$ with $G/\widetilde{G}\cong{SL(2,8)}$.

\begin{lemma}
\label{small 2gp}
Let $G$ be an extension of $SL(2,8)$ by a finite solvable group $K$ 
of even order, and let $L$ be the first term of the derived series for $K$ 
such that $|L^{ab}|$ is even.
If $\dim_{\mathbb{F}_2}L/X^2(L)\leq5$ then $G$ is not perfect.
\end{lemma}

\begin{proof}
The group $SL(2,8)$ is superperfect: $H^i(SL(2,8);\mathbb{Z})=0$ for $i=1,2$.
Since $K/L$ has odd order,  $H^i(K/L;\mathbb{F}_2)=0$ for all $i>0$.
An application of the Lyndon-Hochschild-Serre spectral sequence shows that 
 $H^i(G/L;\mathbb{F}_2)=0$ for $i=1,2$.
Since $\Phi_9(X)=X^6+X^3+1$ has no proper factors in $\mathbb{F}_2[X]$,
the group $SL(5,2)$ has no element of order 9.
Since $SL(2,8)$ has such an element and is simple,
there are no nontrivial homomorphisms from $G$ to $SL(5,2)$.
Therefore if $r\leq5$ then $L/X^2(L)$ is central in $G/X^2(L)$,
and so $G/X^2(L)\cong{G/L}\times{L/X^2(L)}$.
Since $|L^{ab}|$ is even, $L/X^2(L)\not=1$, 
and so $G$ is not perfect.
\end{proof}

\begin{lemma}
\label{noC2rxC7}
If $p$ is an odd prime such that $p\equiv3$ or $5~mod~(7)$
then $GL(6,p)$ has no subgroup $P\cong{C_2^r}\rtimes{C_7}$ with $r>1$.
\end{lemma}

\begin{proof}
Suppose that $P<GL(6,P)$.
The subgroup $C_2^r$ is conjugate to a diagonal subgroup $\Delta$.
(For instance, 
we may use the fact that $\mathbb{F}_p[C_2^c]$ is a semisimple ring,
since $p\not=2$, and the simple modules are 1-dimensional.)
The normalizer of $\Delta$ permutes the diagonal entries, 
and so maps to the symmetric group $S_6$, and the kernel of $N_{SL(6,p)}(\Delta)\to{S_6}$ centralizes $\Delta$.
The centralizer $C_\Delta=C_{GL(6,p)}(\Delta)$ is a subgroup of 
$(\mathbb{F}_p^\times)^r\times{GL(6-r,p)}$.
However, if $p\equiv3$ or $5~mod~(7)$ and $r>0$ 
then the order of $GL(6-r,p)$ is prime to 7.
Since $S_6$ has no element of order 7, 
it follows that $SL(6,p)$ has no subgroup of the form
$C_2^r\rtimes{C_7}$ with $r>0$.
\end{proof}

Let $p$ be an odd prime and $H$ a Sylow 2-subgroup for $GL(2,p)$.
Then $H^3$ embeds as a block diagonal subgroup of $GL(6,p)$.
The 2-group ${H^3}\rtimes{C_2}$ obtained by
adjoining the permutation matrix corresponding to $(3,5)(4,6)$
has odd index in $GL(6,p)$, by comparison of orders, and so
is a Sylow 2-subgroup for $GL(6,p)$.
In the next theorem we shall choose $p$ to be $\equiv3~mod~(8)$, 
for convenience,  as $H$ then has order 16.

\begin{theorem}
\label{no sl(2,8)}
If $G$ is a finitely generated, perfect, virtually solvable subgroup of $GL(6,\mathbb{Q})$
such that $G/\widetilde{G}$ is a minimal simplegroup
then $G/\widetilde{G}\cong{A_5}$ or $PSL(2,7)$.
\end{theorem}

\begin{proof}
It is enough to show that $G/\widetilde{G}$ cannot be $SL(2,8)$,
since $|G|$ divides $2^{12}3^45.7$, by Lemma \ref{sl2}.
Assume the contrary.
Then for sufficiently large $p$ the finite group $SL(6,p)$ has a subgroup $J$ 
such that $J/\widetilde{J}\cong{SL(2,8)}$.
Let $P$ be a Sylow 2-subgroup of $J$.
Then $P/P\cap\widetilde{J}$ is a Sylow 2-subgroup of $SL(2,8)$,
and so $P/P\cap\widetilde{J}\cong{C_2^3}$.
The Borel subgroups of $SL(2,8)$ are isomorphic to $C_2^3\rtimes{C_7}$.
The preimage of such a subgroup in $J$ is solvable, 
and so has a $\{2,7\}$-Hall subgroup $H$.
If $p\equiv3$ or 5 $mod~(7^2)$ then $|SL(6,p|$ is not divisible by $7^2$,
and so $H\cong{P}\rtimes{C_7}$.

If $X^2(P)\not=P\cap\widetilde{J}$ then $L/X^2(L)\not=1$ and
a Sylow 2-subgroup of $SL(2,8)$ acts trivially on $L/X^2(L)$,
which is normal in $J/X^2(L)$.
Hence $SL(2,8)$ acts trivially and the extension splits,
by the argument of Lemma \ref{small 2gp}.
Therefore we may assume that $P/X^2(P)\cong{C_2^3}$.
It then follows from the lemma that either $|\widetilde{J}|$ is odd or 
$X^2(P)/X^2X^2(P)\cong{L/X_2(L)}\cong{C_2^r}$ for some $r\geq6$.

Let $p$ be a prime such that $p\equiv3~mod~(8)$. 
Then the 2-Sylow subgroups of $SL(6,p)$ have order $2^{12}$ and 
abelianization $C_2^4$.
It may be verified that if $S$ is such a group then
$X^2(S)\cong{C_4}\times(C_8\oplus{C_4})\rtimes{C_2}$.
Hence every subgroup of $X^2(S)$ can be generated by 4 elements.
Since $P$ is isomorphic to a subgroup of $S$ it follows that 
$X^2(P)$ is isomorphic to a subgroup of $X^2(S)$,
and so $X^2(P)/X^2X^2(P)\cong{C_2^s}$ for some $s\leq4$.
Hence $X^2(P)=1$, by Lemma \ref{small 2gp},
and so $H\cong{C_2^3\rtimes{C_7}}$.

If we now assume also that $p\equiv3~mod~(7)$
(so that $p\equiv3~mod~(56)$) 
then Lemma \ref{noC2rxC7} gives a contradiction.
Hence our assumption was wrong and $GL(6,\mathbb{Q})$ has no such subgroup $G$.
\end{proof}

This result is again sharp, as $SL(2,8)$ embeds in $GL(7,\mathbb{Z})$ \cite{HP}.
It shall be used in Lemmas \ref{mfixed} and \ref{gl(10,Q)}
and Theorem  \ref{vsolvthm} of the next section.
(We shall thereby almost eliminate $SL(2,8)$ from our considerations.)

\section{ $S$ virtually solvable but not virtually nilpotent}

Let $S$ be a minimal TFNS group which is not virtually abelian.
Then $S$ has solvable normal subgroups $V\leq{U}\leq\widetilde{S}$ such that
$S/U$ is crystallographic,  $U/V$ is finite, $V/I(V)$ has positive rank
and $U/V$ acts effectively on $V/I(V)$,
by Lemma \ref{overline}.
If $U$ is nilpotent then we may assume $V=U$, and then
$\sqrt{S/U}$ is the unique such subgroup of minimal index,
but otherwise $S/V$ may have a non-trivial torsion normal subgroup,
and there may several such subgroups.
We shall use commutative algebra to bound $h(V)$ below,
when $S$ is perfect and not virtually nilpotent.

Let $T\leq\widetilde{S}$ be the preimage in $S$ of the translation
subgroup of $S/U$. 
Since $T$ is finitely generated, $U/V$ is finite and $T/U$ is abelian,
$T/V$ is virtually abelian.
Hence $S/V$ has a free abelian normal subgroup $A$ of finite index.
Clearly $A\cong\mathbb{Z}^h$, where $h=h(S/V)=h(S/U)$.
Let $W$ be the preimage of $A$ in $S$.
Then $V<W$ and $W/V\cong{A}$, so $W$ is solvable.
Hence $W\leq\widetilde{S}$ also, and $S/W$ is finite.

The abelianization $V^{ab}$ is a finitely generated $\mathbb{Z}[A]$-torsion module,
and the $\mathbb{Q}[A]$-module $M=\mathbb{Q}\otimes{V}^{ab}$ is 
a finite dimensional $\mathbb{Q}$-vector space.
Hence $\mathbb{Q}[A]/Ann(M)$ is an Artinian ring.
Let $Supp(M)$ be the set of maximal ideals $\mathfrak{m}$ in  $\mathbb{Q}[A]$
which contain $Ann(M)$, and let $M_\mathfrak{m}$ be the
submodule annihilated by a power of $\mathfrak{m}$.
Then $M_\mathfrak{m}/\mathfrak{m}M_\mathfrak{m}$
is a non-trivial vector space over the field 
$\mathbb{K}_\mathfrak{m}=\mathbb{Q}[A]/{\mathfrak{m}}$,
for each $\mathfrak{m}\in{Supp(M)}$,
and  $M=\oplus{M_\mathfrak{m}}$, 
where the summation is over $Supp(M)$.

Conjugation in $S$ induces an action of $S/V$ on $V^{ab}$ and hence on $M$.
This in turn induces an action of $S/W$ as permutations of $Supp(M)$.
Clearly 
\[
|S\mathfrak{m}|\dim_\mathbb{Q}M_\mathfrak{m}\leq\dim_\mathbb{Q}M\leq{h(V)}=h(U),
\quad\mathrm{for~all}~\mathfrak{m}\in{Supp(M)}.
\]
Let $\mathfrak{e}$ be the kernel of the augmentation homomorphism 
$\varepsilon:\mathbb{Q}[A]\to\mathbb{Q}$.
The following lemma is a variation on Hall's criterion \cite[5.2.10]{Ro}.

\begin{lemma}
\label{vHall}
Let $S$ be a finitely generated,  torsion free group with subgroups $U,V,W$ 
as above.
Then $S$ is virtually nilpotent if and only if $U$ is virtually nilpotent 
and $Supp(M)=\{\mathfrak{e}\}$.
\end{lemma}

\begin{proof}
If $S$ is virtually nilpotent then $W/I(V)$ is also virtually nilpotent.
Let $Z_i$ be an ascending central series for a nilpotent subgroup of finite index
in $W/I(V)$.
The  intersections $Z_i\cap(V/I(V))$ give rise to a filtration of $M$ 
with subquotients annihilated by $\mathfrak{e}$.

Conversely, if $Supp(M)=\{\mathfrak{e}\}$ then $M$ has such a filtration.
This determines an ascending series $\{V_i\mid0\leq{i}\leq{n}\}$ for 
$V$ such that $V_0=I(V)$,  
each subquotient $V_{i+1}/V_i$ is central in $W/V_i$ and $V_n=V$. 
Hence $W/I(V)$ is nilpotent.
Since $W$ is torsion free,
it follows that if $U$ is virtually nilpotent then
$W$ is virtually nilpotent, by a mild variation of Hall's criterion \cite[5.2.10]{Ro}.
\end{proof}

We shall assume for the rest of this section that $S$ is finitely generated,
perfect and TFNS,
and that $H=S/\widetilde{S}$ is a minimal simple group.
The symbols $A, M,U,V$ and $W$ retain the above roles.
The image of $S/V$ in $Aut(M)$ is finitely generated, 
perfect and virtually solvable.

\begin{lemma}
\label{mfixed}
Suppose that $\mathfrak{m}$ is fixed by $S$, and let 
$d=[\mathbb{K}_\mathfrak{m}:\mathbb{Q}]$.
If either $H=A_5$ and $d<4$ or $H=PSL(2,7)$ and $d<6$ 
or $H=SL(2,8)$ and $d<7$
then $\mathfrak{m}=\mathfrak{e}$.
\end{lemma}

\begin{proof}
If $\mathfrak{m}$ is fixed by $S$ then $S$ acts on $\mathbb{K}_\mathfrak{m}$.
The action is $\mathbb{Q}$-linear, and so factors through $GL(d,\mathbb{Q})$.
If $d\leq3$ this group has no non-trivial finitely generated perfect subgroups,
by Lemma \ref{sl2},
and so the action is trivial.
Since $S/A$ is perfect, 
$\mathbb{Z}\otimes_{\mathbb{Z}[S/A]}A=0$, 
and so the image of $A$ in $\mathbb{K}_\mathfrak{m}^\times$ is trivial.
Hence $\mathfrak{m}=\mathfrak{e}$.
This settles the case $H\cong{A_5}$. 

Since $PSL(2,7)$ has order a multiple of 7 we may apply the second part of
Lemma \ref{sl2}, and for $SL(2,8)$ we may apply Theorem \ref{no sl(2,8)}.
\end{proof}

The bounds on $d$ here are best possible,
since $A_5, PSL(2,7)$ and $SL(2,8)$ have embeddings in $GL(d,\mathbb{Q})$ with $d=4,6$ and 7, 
respectively.

\begin{lemma}
\label{gl(10,Q)}
Let $d=4$ if $H\cong{A_5}$, 
let $d=6$ if $H\cong{PSL(2,7)}$ and let $d=7$ if $H\cong{SL(2,8)}$.
If $h(U)\leq{d}$ then either $U$ is virtually abelian or $S$ is virtually nilpotent.
\end{lemma}

\begin{proof}
Let $Z=\zeta\sqrt{U}$.
Since $U$ is nontrivial,
$1\leq{h(Z)}\leq{h(U)}\leq{d}$.
If $h(Z)=h(U)$ then $U/Z$ is a torsion group which 
acts effectively on $Z$.
Hence $U/Z$ is finite, 
by a result of Schur \cite[8.1.11]{Ro},
and so $U$ is virtually abelian.

If $h(Z)<d$ then the action of $S$ on $Z$ by conjugation 
must be trivial,  by Lemma \ref{sl2} and the fact that $S$ is perfect.
Hence $Z$ is central in $S$ and so $Z=\zeta{U}$.
Since $\sqrt{U}$ is torsion free each factor of the upper central series
for $\sqrt{U}$ is torsion free, 
by a result of Mal'tsev \cite[5.2.19]{Ro}.
Therefore if $h(Z)<h(U)\leq{d}$ we may repeat the argument,
since $h(\zeta(\sqrt{U/Z})\leq{h(U/Z)}<d$ and 
$\sqrt{U/Z}$ is torsion free,
and we find that $\zeta\sqrt{U/Z}$ is central in $S/Z$.
Let $\zeta_i{U}$ be the $i$th term of the upper central series for $U$
\cite[page 121]{Ro}.
Iterating the argument, we find that $\zeta_eU=\sqrt{U}$ for some $e<d$ 
and $U/\sqrt{U}$ is a torsion group. 
Since $U$ is torsion free and acts trivially on the sections of its ascending
central series it is nilpotent,
by Baer's extension of Schur's Lemma \cite[14.5.1]{Ro},
and so $U=\sqrt{U}=\zeta_eU$.
Since $\zeta_{i+1}U/\zeta_iU$ is central in $S/\zeta_iU$ for all $i<e$
and $S/U$ is virtually abelian,  $S$ is virtually nilpotent.
\end{proof}

The argument shows that if $h(\zeta\sqrt{U})<d$ and $h(U/\zeta\sqrt{U})<d$ 
then $S$ is virtually nilpotent.
In particular,  this is so if $h(U)\leq3$.

\begin{lemma}
\label{mnotfixed}
Let $\mathfrak{m}\in{Supp(M)}$, and suppose that $\mathfrak{m}$ is not fixed by $S$.
Then 
\begin{enumerate}
\item{}if $S/\widetilde{S}\cong{A_5}$ then $h(S)\geq9$,
and if $h(S)\leq15$ then\\ $\dim_\mathbb{Q}M_\mathfrak{m}\leq2$;
\item{}if $S/\widetilde{S}\cong{PSL(2,7)}$ then $h(S)\geq13$,
and if $h(S)\leq15$ then\\ $\dim_\mathbb{Q}M_\mathfrak{m}=1$;
\item{}if $S/\widetilde{S}\cong{SL(2,8)}$ then $h(S)\geq16$.
\end{enumerate}
\end{lemma}

\begin{proof}
Since $S_4$ has no perfect subgroups,
each non-trivial orbit $S\mathfrak{m}$ of the action of $S$ on $Supp(M)$ 
has at least 5 members.

If $S/\widetilde{S}\cong{A_5}$ then $h(S/U)\geq4$ and so $h(S)\geq9$.
If $h(S)\leq15$ then $|S\mathfrak{m}|\dim_\mathbb{Q}M_\mathfrak{m}\leq11$
and so $\dim_\mathbb{Q}M_\mathfrak{m}\leq2$.

If $S/\widetilde{S}\cong{PSL(2,7)}$ then $h(S/U)\geq6$.
There is no nontrivial homomorphism from $S$ to the symmetric group $S_6$,
since $S_6$ has no subgroup with order a multiple of 7.
Hence if $S\mathfrak{m}$ is non-trivial it must have order $\geq7$.
Hence $h(S)\geq{h(S/U)}+|S\mathfrak{m}|\geq13$.
If $h(S)\leq15$ then $|S\mathfrak{m}|\dim_\mathbb{Q}M_\mathfrak{m}\leq9$
and so $\dim_\mathbb{Q}M_\mathfrak{m}=1$.

If $S/\widetilde{S}\cong{SL(2,8)}$ then $h(S/U)\geq7$.
The symmetric group $S_8$ has no subgroup which is an extension of $SL(2,8)$.
Hence if $S\mathfrak{m}$ is non-trivial it must have order $\geq9$,
and so $h(S)\geq{h(S/U)}+|S\mathfrak{m}|\geq16$.
\end{proof}

Together these lemmas give the following bounds.

\begin{theorem}
\label{vsolvthm}
Let $S$ be a finitely generated, perfect group which is TFNS but not virtually nilpotent.
If $H=S/\widetilde{S}\cong{A_5}$ then $h(S)\geq9$,
if $H\cong{PSL(2,7)}$
then $h(S)\geq13$ and if $SL(2,8)$ then $h(S)\geq14$.
\end{theorem}

\begin{proof}
Suppose that $H\cong{A_5}$ and $h(S)\leq8$.
Then $h(U)\leq4$, since $h(S/U)\geq4$.
Hence $U$ is virtually nilpotent, by Lemma \ref{gl(10,Q)},
and $S$ acts trivially on $Supp(M)$, by Lemma \ref{mnotfixed}.
Since $S$ is not virtually nilpotent $Supp(M)$ has a member
$\mathfrak{m}\not=\mathfrak{e}$, by Lemma \ref{vHall},
and $[\mathbb{K}_\mathfrak{m}:\mathbb{Q}]\geq4$, 
by Lemma \ref{mfixed}.
Then $\dim_\mathbb{Q}M=h(U)=h(S/U)=4$, and so $V$ is abelian.
The crystallographic group $S/U$ has an element $s$ of order 5.
The normal closure of $s$ in $S/U$ is the whole of $S/U$, 
since $H$ is simple and $S/U$ is perfect.
Since $U/V$ is finite,  
there is an element $s'\in{S}$ with image $s$ in $S/U$, 
and whose image in $S/V$ has order a power of 5.
Since $S$ is torsion free and $V$ has rank 4 it follows that $s'$ must centralize $V$.
But then the normal closure of $s'$ in $S$ centralizes $V$ also, 
and so $S$ is virtually nilpotent, contrary to hypothesis.

If $H\cong{PSL(2,7)}$ then $h(S/U)\geq6$,
and if $h(S/U)=6$ or 7 then $S/U$ has 7-torsion \cite{Lu23}.
An argument similar to that for $A_5$ shows that if $h(U)\leq6$ 
then $S$ must be virtually nilpotent.

If $H\cong{SL(2,8)}$ then $h(S/U)\geq7$.
Since $S/V$ acts effectively on $V/I(V)$, by Lemma \ref{overline},
$V/I(V)$ has rank $\geq7$, by Theorem \ref{no sl(2,8)}.
Hence $h(U)=h(V)\geq6$ and so $h(S)\geq13$.
\end{proof}

There is an alternative argument for polycyclic groups, 
which sharpens the result for $H\cong{SL(2,8)}$.

\begin{theorem}
\label{vpcnonvnil}
Let $S$ be a virtually polycyclic group which is perfect and TFNS,
but not virtually nilpotent, and let $U$ be a normal subgroup which contains $\sqrt {S}$ 
and such that $S/U$ is crystallographic.
Then $h(S)\geq2h(S/U)+1$, 
with equality only if $\sqrt{S}$ is abelian and $[U:\sqrt{S}]$ is finite.
In particular,
if $H\cong{SL(2,8)}$ then $h(S)\geq15$.
\end{theorem}

\begin{proof}
We may assume that $U$ contains $\sqrt{S}$ as a subgroup of finite index,
since $S/\sqrt{S}$ is virtually abelian.
Then $h(S/U)=h(S/\sqrt{S})<h(U)=h(\sqrt{S}^{ab})$, by Lemma \ref{Wilson}. 
Since $\sqrt{S}$ is torsion free,
$h(\sqrt{S})=h(\sqrt{S}^{ab}$ if and only if $\sqrt{S}$ is abelian.
Thus the first assertion is clear.

If $H\cong{SL(2,8)}$ then $h(S/U)\geq7$,  and so $h(S)\geq15$.
\end{proof}

\section{nilpotent groups}

We shall say that  a finitely generated nilpotent group $N$
 is {\it of type $[m,n]$} if 
$N/\tilde\gamma_2N\cong\mathbb{Z}^m$ and 
$\tilde\gamma_2N/\tilde\gamma_3N\cong\mathbb{Z}^n$.
(Note that $n\leq\binom{m}2$, since $N/\tilde\gamma_3N$
is a quotient of $F(m)$,  the free group of rank $m$.)

Let $\mathbb{Q}N_{i/i+1}=\mathbb{Q}\otimes\tilde\gamma_iN/\tilde\gamma_{i+1}N$,
for $i\geq1$, 
and let $\mathbb{Q}N^{ab}=\mathbb{Q}N_{1/2}$.
Automorphisms of $N$ must preserve the rational {\it commutator pairing}
\[
[-,-]_\mathbb{Q}:\mathbb{Q}N^{ab}\wedge\mathbb{Q}N^{ab}\to
\mathbb{Q}N_{2/3}.
\]
This pairing has two related aspects.
It is a skew-symmetric pairing on $\mathbb{Q}N^{ab}$, 
and also is an epimorphism of $\mathbb{Q}$-vector spaces.
In the latter context we shall use the term {\it commutator epimorphism\/}.
Let 
\[
R(N)=
\{x\in\mathbb{Q}N^{ab}\mid[x,y]_\mathbb{Q}=0,~\forall~y\in\mathbb{Q}N^{ab} \}
\]
be the radical of the commutator pairing,
and let $\overline{\mathbb{Q}N^{ab}}=\mathbb{Q}N^{ab}/R(N)$.

\begin{lemma}
\label{rad dim}
Let $N$ be a finitely generated nilpotent group of type $[m,n]$ with $n>0$.
Suppose that a finite perfect group $H$ acts effectively on $N$ 
and fixes no nontrivial subspace of $\mathbb{Q}N^{ab}$.
Then
\begin{enumerate}
\item$\dim_\mathbb{Q}\overline{\mathbb{Q}N^{ab}}\ge4$ and either $R(N)=0$ 
or $\dim_\mathbb{Q}R(N)\geq4$;
\item$\mathbb{Q}N_{2/3}$ is a direct summand of
$\mathbb{Q}N^{ab}\wedge\mathbb{Q}N^{ab}$.
\end{enumerate}
\end{lemma}

\begin{proof}
The radical $R(N)$ is an $H$-invariant subspace of $\mathbb{Q}N^{ab}$.
Since $\mathbb{Q}[H]$ is a semisimple ring,
$R(N)$ has an $H$-invariant complement in $\mathbb{Q}N^{ab}$, 
which projects isomorphically onto $\overline{\mathbb{Q}N^{ab}}$.
The complement is non-zero, since $N$ is not virtually abelian.

Since finite perfect subgroups of $GL(3,\mathbb{Q})$ are trivial,
by Lemma \ref{sl2},
any $H$-invariant subspace of $\mathbb{Q}N^{ab}$ of dimension $\leq3$
is fixed pointwise.
Hence $\dim_\mathbb{Q}\overline{\mathbb{Q}N^{ab}}\ge4$
and either $R(N)=0$ or $\dim_\mathbb{Q}R(N)\geq4$.

The second assertion is clear, since $\mathbb{Q}[H]$ is semisimple and 
$[-,-]_\mathbb{Q}$ is an epimorphism.
\end{proof}

Assume the hypotheses of the lemma.
The kernel of the commutator epimorphism has dimension 
$\binom{m}2-n$ and is $H$-invariant. 
Hence if $N$ has type $[4,n]$ with $n=3,4$ or 5 then this kernel 
has dimension $3,2$ or 1,  and so $H$ acts trivially on it.
Let $\omega$ be a nonzero 2-form in this kernel.
Since 2-forms determine skew-symmetric pairings on the dual vector space,
we  may choose a basis $\{e_1,e_2,e_3,e_4\}$ for $\mathbb{Q}N^{ab}$, 
so that $\omega$ is one of $e_1\wedge{e_2}$ or 
$e_1\wedge{e_2}+e_3\wedge{e_4}$.
If an automorphism of $N$ fixes $e_1\wedge{e_2}$ then it fixes 
the subspace of $\mathbb{Q}{N^{ab}}$ generated by $e_1$ and $e_2$.
Hence we may assume that $\omega=e_1\wedge{e_2}+e_3\wedge{e_4}$,
and that $H$ does not fix any 2-dimensional subspace of
$\mathbb{Q}N^{ab}$.
In this case the skew-symmetric pairing determined by $\omega$
is non-degenerate and $H$ acts symplectically.

We shall need to know how the exterior squares of faithful representations 
of degree $\leq10$ decompose as a sum of irreducible representations.
This is an easy exercise in comparing characters.  
See \cite[Chapter 2]{Se}.

\medskip
\noindent$A_5$:\quad $\wedge_2\rho_4\cong\rho_6$,
\quad$\wedge_2\rho_5\cong\rho_4\oplus\rho_6$,
\quad$\wedge_2\rho_6\cong\rho_4\oplus\rho_5\oplus\rho_6$,

$\wedge_22\rho_4\cong\mathbf{1}\oplus\rho_4\oplus\rho_5\oplus3\rho_6$,\quad
$\wedge_22\rho_5\cong\mathbf{1}\oplus4\rho_4\oplus2\rho_5\oplus3\rho_6$,

$\wedge_22\rho_6\cong2\mathbf{1}\oplus4\rho_4\oplus6\rho_5\oplus3\rho_6$,\quad
$\wedge_2(\rho_4\oplus\rho_5)\cong5\rho_4\oplus\rho_5\oplus3\rho_6$, 

{and}\quad$\wedge_2(\rho_4\oplus\rho_6)\cong3(\rho_4\oplus\rho_5\oplus\rho_6)$.

\noindent$PSL(2,7)$:\quad
$\wedge_2\tau_{6a}\cong2\mathbf{1}\oplus\tau_{6a}\oplus\tau_7$,\quad
$\wedge_2\tau_{6b}\cong2\mathbf{1}\oplus\tau_{6b}\oplus\tau_7$,

$\wedge_2\tau_7\cong6\tau_{6a}\oplus\tau_7\oplus\tau_8$\quad
{and}\quad$\wedge_2\tau_8\cong6\tau_{6a}\oplus2\tau_7\oplus\tau_8$

\noindent$SL(2,5)$:\quad
$\wedge_2\pi_{8a}\cong3\mathbf{1}\oplus\widehat\rho_4\oplus3\widehat\rho_5\oplus\widehat\rho_6$,\quad{and}\quad
$\wedge_2\pi_{8b}\cong6\mathbf{1}\oplus4\widehat\rho_4\oplus\widehat\rho_6$.

\noindent$SL(2,7)$:\quad
$\wedge_2\xi_8\cong\mathbf{1}\oplus\widehat\tau_{6b}\oplus
2\widehat\tau_7\oplus\widehat\tau_8$.

\noindent$SL(2,8)$:\quad
$\wedge_2\psi_7\cong\psi_{21}$\quad{and}\quad
$\wedge_2\psi_8\cong\psi_7\oplus\psi_{21}$.

\noindent$L_3(2)N2^3$:\quad
$\wedge_2\lambda_{7a}\cong\lambda_{7a}\oplus\lambda_{14}$\quad{and}
\quad $\wedge_2\lambda_{7b}\cong\lambda_{7b}\oplus\lambda_{14}$.

For the representations of degrees $m=11$ to 13 we need only know
that the following exterior squares each have trivial summands of rank $n$, 
for all $n\leq14-m$.

\noindent$A_5$:\quad 
$\wedge_2(3\rho_4)$ and $\wedge_2(2\rho_4\oplus\rho_5)$.

\noindent$PSL(2,7)$:\quad $\wedge_2(\tau_{6i}\oplus\tau_{6j})$ and 
$\wedge_2(\tau_{6i}\oplus\tau_7)$ 

\noindent$SL(2,5)$:\quad$\wedge_2(\pi_{8i}\oplus\widehat\rho_4)$ and
$\wedge_2(\pi_{8i}\oplus\widehat\rho_5)$ 

On the other hand, $\wedge_2(\rho_5\oplus\rho_6)$ has  no trivial summands.

\section{$S$ virtually nilpotent}

In this section we shall assume that $S$ is virtually nilpotent,
but not virtually abelian, and that $S$ is perfect and $h(S)\leq14$.
However we do not assume here that $H=S/\sqrt{S}$ is simple,
and we do not need the notation of Lemma \ref{overline},
as we may take $V=U=I(\sqrt{S})$, by Lemma \ref{vnil}.
Our goal is to limit the possibilities for $H$ and for
the type of $N=\sqrt{S}$.
The fact that $S$ is torsion free is used only through Lemma \ref{vnil},
to show that $H$ embeds in $Aut(\mathbb{Q}N^{ab})$.

We shall play off the $\mathbb{Q}[H]$-module structures of $\mathbb{Q}N^{ab}$
and $\mathbb{Q}N_{2/3}=\mathbb{Q}\otimes\tilde\gamma_2N/\tilde\gamma_3N$ 
against each other.
The key conditions are
\begin{enumerate}
\item$H$ embeds in $Aut(\mathbb{Q}N^{ab})$;
\item$\dim_\mathbb{Q}N^{ab}+\dim_\mathbb{Q}\mathbb{Q}N_{2/3}\leq{h(S)}\leq14$ 
and  $\dim_\mathbb{Q}N_{2/3}\geq1$;
\item$\mathbb{Q}N_{2/3}$ is a $\mathbb{Q}[H]$-summand of $\wedge_2\mathbb{Q}N^{ab}$;
\item{}$\mathbb{Q}$ is not a $\mathbb{Q}[H]$-summand of $\mathbb{Q}N^{ab}$,
by Lemmas \ref{HS?} and \ref{vnil};
\item{}if $\mathbb{Q}$ is a $\mathbb{Q}[H]$-summand of $\mathbb{Q}N_{2/3}$ then
$\mathbb{Q}N^{ab}$ has a symplectic summand, by Lemma \ref{symplectic}.
\end{enumerate}

Let $[m,n]$ be the type of $N$.
Then $m+n\leq{h(S)}\leq14$ and $n>0$, by (2).
If $ \mathbb{Q}$ is a $\mathbb{Q}[H]$-summand of $\mathbb{Q}N_{2/3}$ 
then $\mathbb{Q}N^{ab}\cong{V_0}\oplus{V_1}$,
where $V_1$ supports a nonsingular skew-symmetric pairing,
with even rank $r$. 
Since finite subgroups of $Sp(4,\mathbb{Q})$ are solvable, $r\geq6$,
and since $H$ acts effectively on $V_0$,
either $V_0=0$ or $m-r=\dim_\mathbb{Q}V_0\geq4$.
In the latter case $m\geq10$.
There is always such a summand $\mathbb{Q}$ if $n\leq3$.

On the other hand,
the fact that $A_5$ acts effectively on $F(4)/\gamma_3(F(4))$ 
(see \S8 below) shows that 
$\mathbb{Q}N^{ab}$ need not have a symplectic summand.

\begin{lemma}
\label{Sp12}
If $H=S/\sqrt{S}$ is not a subgroup of $GL(10,\mathbb{Q})$ then 
it is a subgroup of $Sp(12,\mathbb{Q})$.
\end{lemma}

\begin{proof}
Since $h(S)\leq14$ and $S$ is not virtually abelian, 
$11\leq{m}\leq13$.
Hence $n\leq3$ and so  $\mathbb{Q}N_{2/3}$ is a trivial $H$-module.
Let $\lambda:\mathbb{Q}N_{2/3}\to\mathbb{Q}$ be an epimorphism.
Then $H$ preserves the non-zero skew-symmetric pairing 
$\omega=\lambda\circ[-,-]_\mathbb{Q}$,
and $\mathbb{Q}N^{ab}\cong{R(\omega)\oplus{V}}$,
where the induced pairing on $V$ is non-singular,
and so $d=\dim_\mathbb{Q}V$ is even.
Let $\pi_R$ and $\pi_V$ be the projections of $H$ into $GL(r,\mathbb{Q})$ and 
$Sp(d,\mathbb{Q})$, respectively.

If $R(\omega)=0$ then $m=12$ and $H$ is a subgroup of $Sp(12,\mathbb{Q})$,
since $\mathbb{Q}N^{ab}$ is a faithful $H$-module.

If $R(\omega)\not=0$ then $r=\dim_\mathbb{Q}R(\omega)\geq4$ and $d=4$,  
6 or 8,  since $\mathbb{Q}N^{ab}$ has no trivial summands.
Let $\pi_R$ and $\pi_V$ be the projections of $H$ into $GL(r,\mathbb{Q})$ and 
$Sp(d,\mathbb{Q})$, respectively.
Then $K_R=\mathrm{Ker}(\pi_R)$ and $K_V=\mathrm{Ker}(\pi_V)$ 
are each subgroups of $\widetilde{H}$,  
and $K_R\cap{K_V}=1$,
since $\mathbb{Q}N^{ab}$ is a faithful $H$-module.
Since $d\leq8$ and $\pi_V(H)$ is perfect,
$\pi_V(H)$ is either $A_5$, $PSL(2,7)$ or $SL(2,5)$ \cite{Ki}.
In the first two cases $K_V=\widetilde{H}$,
and so $K_R\leq{K_V}$.
Hence $\pi_R$ is injective and $H<GL(r,\mathbb{Z})$.
Otherwise,  $d=8$, so $r=4$ and $\pi_R(H)\cong{A_5}$.
The image $\pi_V(K_R)$ is a solvable normal subgroup of $\pi_V(H)\cong{SL(2,5)}$.
Since $\pi_V$ maps $K_R$ injectively, $K_R=C_2$.
Hence $H\cong{SL(2,5)}$ and $\pi_V$ is injective.
Since $r<10$ and $SL(2,5)<GL(8,\mathbb{Z})$, this proves the lemma.
\end{proof}

The only (minimal perfect) groups with representations satisfying conditions (1)--(5) 
above are $H\cong{A_5}$,  $PSL(2,7)$,  $SL(2,5)$, $SL(2,7)$ or $L_3(2)N2^3$.
Consideration of the decompositions of $\mathbb{Q}N^{ab}$ and 
$\mathbb{Q}N_{2/3}$ as $\mathbb{Q}[H]$-modules with condition (3) 
shows that the possibilities for $[m,n]$ and $H$ are:

$m=4$ and $n=6$. Only $H=A_5$.

$m=5$ and $n=4$ or 6. Only $A_5$.

$m=6$ and $n=1$ or 2. Only $PSL(2,7)$.

$m=6$ and $n=4$ or 5. Only  $A_5$.

$m=6$ and $n=6$.  $A_5$ or $PSL(2,7)$.

$m=6$ and $n=7$ or 8.  Only $PSL(2,7)$.

$m=7$ and $n=6$.  Only $PSL(2,7)$.

$m=7$ and $n=7$.  $PSL(2,7)$ or $L_3(2)N2^3$.

$m=8$ and $n=1$.  $A_5$, $SL(2,5)$ or $SL(2,7)$.

$m=8$ and $n=2$ or 3.  Only $SL(2,5)$.

$m=8$ and $n=4$ or 5.  $A_5$ or $SL(2,5)$.

$m=8$ and $n=6$. $A_5, PSL(2,7)$ or $SL(2,5)$.
 
$m=9$ and $n=4$ or 5.  Only $A_5$.

$m=10$ and $n=1$. Only $A_5$ and $\mathbb{Q}N^{ab}\cong\rho_5\oplus\rho_5$.

$m=10$ and $n=4$. Only $A_5$.

$m=12$ and $n=1$ or 2.  $A_5, PSL(2,7)$ or $SL(2,5)$.

$m=13$ and $n=1$.  $A_5$, $PSL(2,7)$ or $SL(2,5)$.

Parallel arguments apply further down the $\mathbb{Q}$-lower central series,
since conjugation in $S$ induces actions of $H=S/\sqrt{S}$ on each of the subquotients
$\tilde\gamma_iN/\tilde\gamma_{i+1}N$, and 
there are natural epimorphisms from 
$\mathbb{Q}N^{ab}\otimes\mathbb{Q}N_{i/i+1}$ to $\mathbb{Q}N_{i+1/i+2}$,
for all $i\geq1$ \cite[5.2.5]{Ro}. 

\begin{theorem}
\label{vmetab}
Let $S$ be a finitely generated perfect group which is virtually nilpotent and TFNS.
If $h(S)\leq14$ then either $\tilde\gamma_4\sqrt{S}=1$,
or $S/\sqrt{S}\cong{A_5}$,  $\sqrt{S}$ has type $[8,1]$, 
$h(\tilde\gamma_3\sqrt{S})=5$
and $\tilde\gamma_4\sqrt{S}\cong\mathbb{Z}$.
In all cases,  $\sqrt{S}$ is metabelian.
\end{theorem}

\begin{proof}
Let $H=S/\sqrt{S}$, and suppose that $\sqrt{S}$ has type $[m,n]$.
Then $\mathbb{Q}\sqrt{S}_{3/4}$ is a summand of 
$\mathbb{Q}\sqrt{S}^{ab}\otimes\mathbb{Q}\sqrt{S}_{2/3}$, 
of rank $\leq14-m-n$.
Checking the possibilities, we see that either $H\cong{A_5}$ and $[m,n]=[4,6]$,
[5,4], [5,5],  [6,4],  [6,6], [8,1], [8,4] or [9,4],
or $H\cong{PSL(2,7)}$ and $[m,n]=[6,1]$, [6,2] or [6,6].

If $\tilde\gamma_4\sqrt{S}=1$ then $\sqrt{S}$ is metabelian,
since  $G''\leq\gamma_4G$ for any group $G$.
If $\tilde\gamma_4\sqrt{S}\not=1$ then 
$\mathbb{Q}\sqrt{S}_{4/5}$ is a summand of
$\mathbb{Q}\sqrt{S}^{ab}\otimes\mathbb{Q}\sqrt{S}_{3/4}$.
Since $h(S)\leq14$ we must have $H\cong{A_5}$, $[m,n]=[8,1]$,  
${\mathbb{Q}\sqrt{S}_{3/4}=\rho_4}$ and $\mathbb{Q}\sqrt{S}_{4/5}=\mathbf{1}$.
Hence $h(\tilde\gamma_3\sqrt{S})=5$ and $\tilde\gamma_4\sqrt{S}\cong\mathbb{Z}$.
Closer inspection shows that since 
$\tilde\gamma_2\sqrt{S}/\tilde\gamma_3\sqrt{S}$ is cyclic,
$I(\sqrt{S})'\leq\tilde\gamma_5\sqrt{S}=1$, and so $\sqrt{S}$ is again metabelian.
\end{proof}


\section{torsion in the crystallographic quotients}

In this section we shall use the fact that the crystallographic quotients of our groups 
often have ``large" finite subgroups to reduce the list of unsettled cases 
with Hirsch length $h<15$ further.

We shall assume henceforth that $S$ is finitely generated,  perfect and TFNS.
Let $\langle\langle{g}\rangle\rangle$ be the normal closure of $g$ in $S$,
and let $S_1=S/\langle\langle{g}\rangle\rangle$.
If $g\notin\widetilde{S}$ then $S_1/\widetilde{S_1}$ is a proper quotient 
of $S/\widetilde{S}$, and so $S_1=\widetilde{S_1}$.
Hence $S_1$ is a perfect solvable group.
Thus $S_1=1$ and so any such element $g$ normally generates $S$.

\begin{lemma}
\label{I(N)<zeta}
Suppose that $G$ has normal subgroups $K\leq{J}$ such that $J/K$ is finite
and $K$ is torsion free abelian of rank $n$.
If $G$ has an element $g$ which normally generates $G/J$ and
whose image in $G/J$ has prime order $p>n$ then $K\leq\zeta{G}$.
\end{lemma}

\begin{proof}
Let $g\in{G}$ be an element whose image in $G/J$ has order $p$,
and which normally generates $G/J$.
Let $[J:K]=p^dq$ where $(p,q)=1$ and $d\geq0$.
After replacing $g$ by $g^q$, if necessary, 
we may assume that the image of $g$ in $G/K$ has order $p^e$ for some $e\geq1$.
Since $p>n$ and the subgroup generated by $K$ and $g$ is torsion free, 
$g$ centralizes $K$.
Since $K$ is normal in $G$,
the normal closure of $g$ in $G/J$ centralizes $K$, 
and since $K$ is an abelian normal subgroup it follows that $J\leq\zeta{G}$.
\end{proof}

We may now exclude some of the cases with $S$ virtually nilpotent of type $[m,n]$, 
$I(\sqrt{S})\cong\mathbb{Z}^n$ and $H=S/\sqrt{S}$ simple.
If $p=5$, $H=A_5$ and $[m,n]=[5,4]$, [8,4] or [9,4],
or if $p=7$, $H\cong{PSL(2,7)}$ and $[m,n]=[6,6]$ or [7,6]
then $S/I(\sqrt{S})$ has $p$-torsion \cite[Chapter 6]{HP} and $n<p$.
Since $I(\sqrt{S})$ is a non-trivial $H$-module,  by Lemma \ref{rad dim},
it follows from Lemma \ref{I(N)<zeta} that $S$ cannot be torsion free.

Let $H$ be a finite group which acts effectively on an abelian group $A$
and let $G$ be an extension of $H$ by $A$ corresponding to $\xi\in{H^2(H;A)}$.
If $K<H$ has order relatively prime to that of $H^2(H;A)$ 
then the restriction to $K$ of $c(\xi)$ is 0,
and so the restricted extension splits.

\begin{lemma}
\label{split normal}
Let $G$ be a crystallographic group with translation subgroup $A$ and holonomy $H$.
Suppose that the Sylow $p$-subgroup of $H$ is a cyclic subgroup $C$ which 
acts without fixed points on $A$.
Then $G$ has a subgroup isomorphic to $N_H(C)$.
\end{lemma}

\begin{proof}
Let $D=N_H(C)/C$.
Then $H^1(D;H^1(C;A))=0$,
since $H^1(C;A)$ is a finite $p$-group and the order of $D$ is prime to $p$.
Since $C$ is cyclic and acts on $A$ without fixed points,  $H^2(C;A)=H^0(C;A)=0$.
The LHS Spectral sequence gives $H^2(N_H(C);A)=0$,
and so the projection of $G$ onto $H$ splits over the subgroup $N_H(C)$.
\end{proof}

Suppose that $S$ has a normal subgroup $N$ such that $I(N)$ is abelian and
$G=S/I(N)$ is crystallographic with holonomy $H=S/N$,
and that $H$ has a cyclic Sylow $p$-subgroup $C$.
If $h(S/I(N))=p-1$ then elements of $H$ of order $p$ act on $A=N/I(N)$
without fixed points.
We shall consider each of the cases $H=A_5$,  
$SL(2,5)$, $PSL(2,7)$, 
$SL(2,7)$ and $L_3(2)N2^3$ in turn.
(The Sylow 3- and 7-subgroups of $SL(2,8)$ are cyclic, 
but have fixed points in the representations $\psi_7$ and $\psi_8$,
so Lemma \ref{split normal} does not apply.)

\medskip
\noindent$H=A_5${\bf :}
If $S$ is virtually nilpotent,  $N=\sqrt{S}$, 
$H\cong{A_5}$ and $h(N/I(N))=8$ then 
$\mathbb{Q}N^{ab}=\rho_4\oplus\rho_4$.
Let $C$ be a a Sylow 5-subgroup of $H$.
Then $C\cong{C_5}$  and $N_H(C)\cong{D_{10}}$.
As a $\mathbb{Z}[C]$-module $N/I(N)$ is a direct sum $L\oplus{L'}$, 
where $L$ and $L'$ are irreducible and of rank 4 as abelian groups.
Hence they are each either the augmentation ideal in $\mathbb{Z}[C]$
or its $\mathbb{Z}$-linear dual. 
In either case,  $C$ acts  on $N/I(N)$ without fixed points,
and so $D_{10}<S/I(N)$,
by Lemma \ref{split normal}.
There are no  5-dimensional Bieberbach groups with holonomy
mapping onto $D_{10}$ \cite{CS},
and so we cannot have $h(I(N))\leq5$.
Hence we may exclude the cases with $S/\sqrt{S}\cong{A_5}$ and $m=8$,
excepting perhaps when $h(S)=14$ and $[m,n]=[8,1]$ or [8,6].
However the case $[m,n]=[8,1]$ follows on first using Lemma \ref{split normal}
to show that $\overline{S}=S/\tilde\gamma_3\sqrt{S}$ 
has a subgroup isomorphic to $D_{10}$.

We may also sharpen one part of Theorem \ref{vsolvthm}.

\begin{theorem}
\label{A5h>9}
If $S/\widetilde{S}\cong{A_5}$ then $h(S)\geq10$.
\end{theorem}

\begin{proof}
Let $V\leq{U}\leq\widetilde{S}$ be normal subgroups of $S$ 
as in Lemma \ref{overline}.
Suppose that $h(S)<10$. 
We may assume that $h(S)=9$ and that $S$ is not virtually nilpotent, 
by Theorem \ref{vsolvthm} and the results of \S8 above.
Moreover $h(U)\geq4$, with equality only if $U$ is virtually abelian,
by Lemma \ref{gl(10,Q)}.
Hence $h(S/U)=4$ or 5,  and so $S/U$ has 5-torsion \cite{Lu23}.

If $h(S/U)=5$ then $h(U)=4$.
Since $h(\zeta\sqrt{U})>0$,  
it follows from Lemma \ref{gl(10,Q)} and the subsequent remark 
that $\sqrt{U}$ is abelian.
Hence $\sqrt{U}$ is central in $S$,
by Lemma \ref{I(N)<zeta}, and so $S$ is virtually nilpotent.

Therefore $h(S/U)=4$ and $h(U)=5$.
The Sylow 5-subgroups of $A_5$ are cyclic and have normalizer $D_{10}$.
Hence $S/U$ has a subgroup isomorphic to $D_{10}$, 
by Lemma \ref{split normal}.
Let $X$ be the preimage of this subgroup in $S$.
There are no  5-dimensional Bieberbach groups with holonomy $X/V$
mapping onto $D_{10}$ \cite{CS},
and it follows easily that there are no torsion free extensions of 
$X/V$ by an abelian group of rank 5.
Therefore $V$ cannot be abelian of rank 5, and so $I(V)\not=1$.

If $h(I(V))=1$ then $I(V)$ is central in $S$,  and so $V$ is nilpotent.
Since $S$ is not virtually nilpotent and $S/U$ has no nontrivial normal subgroup
of infinite index,  $\sqrt{S}\leq{U}$.
Since $V\leq\sqrt{S}$ and $U/V$ acts effectively on $V/I(V)$ we have $V=\sqrt{S}$.
Since $S$ is not virtually nilpotent $V/I(V)$ is not central in $S/I(V)$.
Hence the subgroup $D_{10}$ of $S/U$ must act effectively.
The action is compatible with the rational commutator pairing.
Since $V$ is nonabelian the radical $R(V)$ has rank 2 or 0.
It is easily seen that $R(V)$ must be 0, and so the pairing is nondegenerate.
Hence the action of $D_{10}$ is symplectic.
But $D_{10}$ is not a subgroup of  $Sp(4,\mathbb{Q})$ \cite[Chapter 4]{Ki}.
Therefore there is no such group.

If $h(I(V))>1$ then $h(V/I(V))<4$,
and so $V/I(V)$ is central in $S/I(V)$, by Lemma \ref{I(N)<zeta}.
Hence $V=U$, since $U/V$ acts effectively on $V/I(V)$.
Since $S$ is not virtually nilpotent, 
$I(V)$ must be abelian of rank 4 and so $V/I(V)$ is abelian of rank 1,
by the argument of Lemma \ref{gl(10,Q)}.
Hence $X/I(V)$ is a central extension of $D_{10}$ by a torsion free abelian group 
of rank 1.
Since $H^2(D_{10};\mathbb{Z})\cong\mathbb{Z}/2\mathbb{Z}$, 
this extension splits over $C_5$ and so $X/I(V)$ has a subgroup of order 5.
Since $S$ is torsion free and $I(V)$ is abelian of rank 4, 
this subgroup of order 5 must act trivially on $I(V)$.
But then $I(V)$ is central, and $S$ is virtually nilpotent.
Thus we again reach a contradiction.
\end{proof}

\noindent$H=SL(2,5)${\bf :}
Consideration of the character tables shows that
the faithful 8-dimensional representations of $SL(2,5)$ 
each restrict to fixed-point free representations of the Sylow 5-subgroups.
Hence the corresponding crystallographic groups each have subgroups isomorphic 
to  the normalizers of these subgroups, which are metacyclic of order 20.
(These are the Borel subgroups of $SL(2,5)$.)
No 6-dimensional Bieberbach group has such a holonomy group  \cite{CS},
and so we may exclude the cases with $S/\sqrt{S}\cong{SL(2,5)}$ and $m=8$.
Similarly for the faithful 12-dimensional representations of $SL(2,5)$,
and for the 12-dimensional representations of $A_5$ with character $3\rho_4$.

If $S/\sqrt{S}\cong{SL(2,5)}$ and $N=\sqrt{S}$ is of type [13,1] then
$\mathbb{Q}N^{ab}\cong\pi_{8i}\oplus\hat\rho_5$,  for some $i=a$ or $b$,
and $SL(2,5)$ acts symplectically.
Hence $h(\zeta{N})=6$.
Let $\overline{N}$ be the quotient of $N/\zeta{N}$ by its torsion subgroup.
Then $\mathbb{Q}\overline{N}\cong\pi_{8a}$ or $\pi_{8b}$,
and so $SL(2,5)$ acts effectively on $\overline{N}$.
As before, $S/\zeta{N}$ must have a subgroup isomorphic to a Borel
subgroup of $SL(2,5)$.
Since $S$ is torsion free, we may exclude this case also.

\medskip
\noindent$H=PSL(2,7)${\bf :}
Lemma \ref{split normal} also applies when $S/\sqrt{S}\cong{PSL(2,7)}$,
$h(S/I(\sqrt{S}))=6$ and $p=7$, 
with $D=M_{7,3}$, the metacyclic group of order 21.
(This is the image of a Borel subgroup of $SL(2,7)$.)
R. Lutowski has used CARAT to verify that $M_{7,3}$ is not 
the group of any 8-dimensional Bieberbach group \cite{Lu23}.
This shall enable us to substantially reduce the role of $PSL(2,7)$
in answering our question.

\begin{theorem}
\label{no7}
If $S/\sqrt{S}\cong{PSL(2,7)}$,
$h(S/I(\sqrt{S}))=6$ and $I(\sqrt{S})\cong\mathbb{Z}^n$, for some $n$,
then $h(S)\geq15$.
\end{theorem}

\begin{proof}
The group $S/I(\sqrt{S})$ has a subgroup $L\cong{M_{7,3}}$, 
by Lemma \ref{split normal}.
Let $W$ be the preimage of $L$ in $S$.
No extension of $L$ by $\mathbb{Z}$ or $\mathbb{Z}^2$ is torsion free.
Hence $n\geq6$ and $\mathbb{Q}\sqrt{S}_{2/3}$ has a summand
which is a faithful $S/\sqrt{S}$-representation.
Hence $L$ acts effectively on $I(\sqrt{S})\cong\mathbb{Z}^n$,
and so $W$ is a Bieberbach group.
Since $M_{7,3}$ is not the group of any 8-dimensional Bieberbach group,
$n\geq9$,  and so $h(S)\geq15$.
\end{proof}

In particular,  if $S/\sqrt{S}\cong{PSL(2,7)}$
then $[m,n]\not=[6,1]$,  [6,2],  [6,6], [6,7], or [6,8].
An argument parallel to the one above for the case with $H\cong{A_5}$ and 
$h(S/I(\sqrt{S}))=8$ shows that we may extend Theorem \ref{no7} to exclude 
the cases with $S/\sqrt{S}\cong{PSL(2,7)}$ and $[m,n]=[12,1]$ or [12,2].
(If  $m=12$ then 
$\mathbb{Q}\sqrt{S}^{ab}=\tau_{6i}\oplus\tau_{6j}$ for some $i,j\in\{a,b\}$.)

Theorem \ref{no7} may be extended to the case when $S$ is not virtually nilpotent.
We shall use the result of Lutowski \cite{Lu23} again, 
together with the simpler observation
that 9 is the smallest dimension of a Bieberbach group with holonomy
cyclic  of order 21 \cite{KP}.
The argument for the next theorem is parallel to that for Theorem \ref{A5h>9}.

\begin{theorem}
\label{vmetabPSL(2,7)}
Let $S$ be a minimal TFNS group with $S/\widetilde{S}\cong{PSL(2,7)}$.
Then $h(S)\geq14$.
\end{theorem}

\begin{proof}
There are normal subgroups $V\leq{U}<S$ such that $S/U$ is crystallographic
and $U/V$ is a finite solvable group which acts effectively on $V/I(V)$,
by Lemma \ref{overline}.
We may assume that $S$ is not virtually nilpotent,
by the observations following Theorem \ref{no7}.
We may assume also that $h(S)=6$, for otherwise $h(S)\geq14$,
by Theorem \ref{vsolvthm}.
Then $S/U$ has  a subgroup isomorphic to $M_{7,3}$, 
by Lemma \ref{split normal}.
Let $X$ be the preimage of this subgroup in $S$.
Then $X/V$ is a finite solvable group.

Suppose that $V$ is abelian and $h=h(V)<9$.
Then $X/V$ acts effectively on $V$,
for otherwise some element of $X$ with non-trivial image in $S/U$ would centralize $V$.
But any such element must map non-trivially to the simple quotient $PSL(2,7)$,
and it would follow that $V$ must be central in $S$.
In particular, $S$ would be virtually nilpotent, 
and the action of the holonomy on $V$ would be trivial.
Hence $U=V\cong\mathbb{Z}$ or $\mathbb{Z}^2$, 
since $h(S/U)=6$ and $S/U$ has holonomy $PSL(2,7)$.
But no extension of a finite nonabelian group by $\mathbb{Z}$ or $\mathbb{Z}^2$
is torsion free.
Hence we may assume that $X/V$ embeds in $GL(8,\mathbb{Q})$.

Since $X/V$ is finite it preserves a lattice $L\leq{V}$.
Hence $X$ has a finitely generated subgroup $W$ which is an extension of $X/V$
by $L\cong\mathbb{Z}^h$.
If $W$ were torsion free then this subgroup would be a Bieberbach group 
of dimension $h\leq8$, and with holonomy of order a multiple of 21,
since $X/V$ acts effectively and maps onto $M_{7,3}$.
But there are no such Bieberbach groups. 
Hence either $I(V)\not=1$ or $h(S)\geq15$.

Thus we may assume that $h(S)=13$ and $I(V)\not=1$.
The argument for Lemma \ref{gl(10,Q)} may be used to show 
that $\sqrt{S}$ is abelian of rank 6.
Hence $h(V/I(V))=1$, and so $U=V$, by Corollary 9.
Since $Aut(V/I(V))$ is abelian, $V/I(V)$ is central in $S/I(V)$,
and so $X/I(V)$ is a central extension of $M_{7,3}$ 
by a torsion free abelian group of rank 1.
Since $H^2(M_{7,3};\mathbb{Z})\cong\mathbb{Z}/3\mathbb{Z}$, 
this extension splits over $C_7$.
Hence $X/I(V)$ has a subgroup of order 7.
This subgroup of order 7 must act trivially on $I(V)$,
since $S$ is torsion free and $I(V)$ is abelian of rank 6.
But then $I(V)$ is central in $S$, and $S$ is virtually nilpotent.
This contradicts our earlier work, and so we cannot have $h(S)<14$.
\end{proof}

The argument of the final paragraph  leads to a similar contradiction if $I(V)$ 
is abelian of rank 6, $V/I(V)$ has rank 2  and $h(S)=14$. 


\medskip
\noindent$H=SL(2,7)$ or $L_3(2)N2^3${\bf :}
Every 8-dimensional crystallographic group with holonomy $SL(2,7)$
is a semidirect product $\mathbb{Z}^8\rtimes_\theta{SL(2,7)}$,
for some effective action $\theta$ \cite[page 295]{HP},
and so we may exclude the case with $S/\sqrt{S}\cong{SL(2,7)}$ and $[m,n]=[8,1]$.
Similarly, the cohomology classes corresponding to extensions
of $L_3(2)N2^3$ by $\mathbb{Z}^8$ which are crystallographic 
have order $\leq2$ \cite[page 298]{HP}.
Hence such extensions split over subgroups of $L_3(2)N2^3$ of odd order.
Since $M_{7,3}$ is such a subgroup, 
and is not the holonomy of a 7-dimensional Bieberbach group,
we may exclude the case with $S/\sqrt{S}\cong{L_3(2)N2^3}$ and $[m,n]=[7,7]$.

\medskip
In the light of the above arguments we find that in all cases
${h(S)\geq10}$ and $S/\widetilde{S}\cong{A_5}$,  $PSL(2,7)$ or $SL(2,8)$.
If $S$ is virtually nilpotent and $h(S)\leq14$ then the list of possibilities for 
the type $[m,n]$ and $H=S/\sqrt{S}$ reduces to

$[m,n]=[4,6]$.  $H\cong{A_5}$ and $\mathbb{Q}\sqrt{S}_{3/4}=\rho_4$ or 0.

$[m,n]=[5,6]$.  $H\cong{A_5}$. 

$[m,n]=[6,4]$.   $H\cong{A_5}$ and $\mathbb{Q}\sqrt{S}_{3/4}=\rho_4$ or 0.

$[m,n]=[6,5]$.   $H\cong{A_5}$.

$[m,n]=[6,6]$.  $H\cong{A_5}$ and $\mathbb{Q}\sqrt{S}_{3/4}=\mathbf{2}$, 
$\mathbf{1}$ or 0.

$[m,n]=[7,7]$.  $H\cong{PSL(2,7)}$.

$[m,n]=[8,6]$.  $H\cong{A_5}$.

$[m,n]=[8,6]$.  $H\cong{PSL(2,7)}$.

$[m,n]=[9,5]$.  $H\cong{A_5}$.

$[m,n]=[10,1]$.  $H\cong{A_5}$ and $\mathbb{Q}\sqrt{S}^{ab}\cong\rho_5\oplus\rho_5$.

$[m,n]=[10,4]$.  $H\cong{A_5}$.

$[m,n]=[12,1]$.  $H\cong{A_5}$, $\mathbb{Q}\sqrt{S}^{ab}\cong\rho_6\oplus\rho_6$
and $\mathbb{Q}\sqrt{S}_{3/4}=\mathbf{1}$ or 0.

$[m,n]=[12,2]$.  $H\cong{A_5}$ and $\mathbb{Q}\sqrt{S}^{ab}\cong\rho_6\oplus\rho_6$.

$[m,n]=[13,1]$.  $H\cong{A_5}$ or $PSL(2,7)$.\\
Thus if $S$ is virtually nilpotent and $h(S)<14$ then $S/\sqrt{S}\cong{A_5}$,
and $\tilde\gamma_3\sqrt{S}=1$ in all cases except for $[m,n]=[4,6]$, [6,4], [6,6]
or [12,1].

If $S$ is not virtually nilpotent and $H\cong{PSL(2,7)}$  
or $H\cong{SL(2,8)}$ then $h(S)\geq14$, 
by Theorems \ref{vsolvthm} and \ref{vmetabPSL(2,7)}.
In the latter case $S$ is not virtually polycyclic, by Theorem \ref{vpcnonvnil}.

We cannot exclude type $[4,6]$ by arguments involving just the lower central series.
Let $\{w,x,y,z\}$ be a basis for $F(4)$,
and define endomorphisms $\sigma$ and $\tau$ by
$\sigma(w)=w^{-1}$, $\sigma(x)=wxy$, $\sigma(y)=y^{-1}$,
$\sigma (z)=yz$ and $\tau(w)=x^{-1}$, $\tau(x)=xyz$, $\tau(y)=z^{-1}$ and
$\tau(z)=(wxy)^{-1}$.
These are automorphisms,
since $\sigma^2=\tau^3=(\sigma\tau)^5=1$,
and define a monomorphism $\theta:A_5\to{Aut(F(4))}$,
with rational abelianization $\rho_4$ in $GL(4,\mathbb{Q})$.
Let $S$ be an extension of $A_5$ by $N=F(4)/\gamma_3F(4)$,
with action $\theta$.
Then $\sqrt{S}=F(4)/\gamma_3F(4)$ is of type [4,6].

If $\alpha:H\to{Aut(N)}$ is a homomorphism then the semidirect product
$N\rtimes_\alpha{H}$ is a basepoint for the set of extensions of $H$ by $N$ with
outer action corresponding to $\alpha$, and so determines a natural bijection from
$H^2(H;\zeta{N})$ to the set of such extensions.
The restriction of an extension $\xi$ to a subgroup $J<H$ splits
if and only if the corresponding cohomology class $c(\xi)$ restricts to 0
in $H^2(J;\zeta{N}^{\alpha|_J})$.
(This is not clear if the outer action does not factor though $Aut(N)$!)

We may apply this observation to $S$ and to $J=A_4<A_5$.
Since $H^2(A_5;\zeta\sqrt{S})$ has order 5 \cite[page 273]{HP} and $(|A_4|,5)=1$, 
the preimage of $J$ in $S$ is a semidirect product,
and so $S$ has torsion.
Taking into account the Jacobi identities,
we see that $\mathbb{Q}F(4)_{3/4}\cong\rho_4\oplus2\rho_5\oplus\rho_6$,
and so $F(4)$ has a canonical $\theta$-invariant subgroup $K$ such that
$\gamma_4F(4)<K<\gamma_3F(4)$ and $h(F(4)/K)=14$.
A similar argument then shows that any extension of $A_5$ by $F(4)/K$ 
with outer action induced by $\theta$ must have 2-torsion.
However we do not know whether such arguments apply
to other virtually nilpotent groups $S$ with $\sqrt{S}$ of type [4,6].

\section{some questions}

1) is there a minimal TFNS group $S$ with  $S/\widetilde{S}\cong{SL(2,8)}$ and 
$h(S)=14$? If so, $S$ must be an extension of a crystallographic group $S/U$ 
with $h(S/U)=7$ by a virtually abelian group $U$ which is not finitely generated.

2) Is there a minimal TFNS group which is an extension of a crystallographic group by 
$\mathbb{Z}$? 
In particular, does every perfect 12-dimensional crystallographic group 
with holonomy $A_5$ have either $D_{10}$ or $A_4$ as a subgroup?

3) Let $S$ be a finitely generated group which is an extension of 
a finite simple group $G$ by a nilpotent group $N$ with torsion free centre $A$, 
and such that $S/A$ is a crystallographic group.
Conjugation in $S$ induces an action $\theta:G\to{Aut}(A)$,
since $A$ is central in $N$.
Suppose $S/A$ has $p$-torsion,
for some prime $p$ which does not divide the order of  $H^2(G;A)$.
Must $S$ have $p$-torsion?
(Note that the class of $S$ as an extension of $S/A$ by $A$ is not induced 
from $H^2_\theta(G;A)$. 
In fact it has infinite order,
since it restricts non-trivially to $H^2(N/A;A)$.)

4) If $H$ is the holonomy group of an $n$-dimensional infranilmanifold 
is it also the holonomy group of a flat $n$-manifold?
This is so if $n\leq4$, by inspection of the known groups.
In general, $H$ is the holonomy of a crystallographic group in dimension $\leq{n}$, 
by Lemma \ref{vnil}.

(The converse fails for $n=4$, since $A_4$ is the holonomy group of a flat 4-manifold,
but not of any other 4-dimensional infranilmanifold.)

5) If $S$ is a minimal TFNS group must it be virtually polycyclic?

6) If $S$ is a minimal TFNS group must it be perfect?

\medskip
If  questions 2) and 3) both have positive answers then there are 
no virtually nilpotent TFNS groups $S$ with $h(S)<14$.

\end{document}